\documentclass[10pt,a4paper]{amsart}

\theoremstyle{plain}
\newtheorem{theorem}{Theorem}[section]
\newtheorem{lemma}[theorem]{Lemma}
\newtheorem{proposition}[theorem]{Proposition}
\newtheorem{corollary}[theorem]{Corollary}
\theoremstyle{definition}
\newtheorem{definition}[theorem]{Definition}

\theoremstyle{remark}
\newtheorem{remark}[theorem]{Remark}
\newtheorem{claim}{Claim}

\newtheorem{notations}[theorem]{Notations}

\newcommand{\ic}{\ensuremath{\mathcal{I}}}
\newcommand{\oc}{\ensuremath{\mathcal{O}}}

\newcommand{\fc}{\ensuremath{\mathcal{F}}}

\newcommand{\ec}{\ensuremath{\mathcal{E}}}
\newcommand{\hc}{\ensuremath{\mathcal{H}}}

\newcommand{\pc}{\ensuremath{\mathcal{P}}}

\newcommand{\Ph}{\check{\mathbb{P}}}

\newcommand{\Pq}{\mathbb{P}^4}
\newcommand{\Pcq}{\mathbb{P}^5}
\newcommand{\Psx}{\mathbb{P}^6}

\newcommand{\Pn}{\mathbb{P}^n}

\newcommand{\Ps}{\mathbb{P}}

\newcommand{\bV}{\mathbb{V}}
\newcommand{\bI}{\mathbb{I}}

\newcommand{\bK}{\mathbb{K}}

\newcommand{\tg}{\tilde{G}}

\def\bin #1#2 {\left( \matrix { #1 \cr #2 \cr } \right) }

\newcommand{\tT}{\tilde{T}}

\begin{document}

\title[ A speciality Theorem for curves in $\Pcq$ ]
{A speciality Theorem for curves in $\Pcq$ contained in
Noether-Lefschetz general fourfolds}

\author{Vincenzo Di Gennaro }
\address{Universit\`a di Roma \lq\lq Tor Vergata\rq\rq, Dipartimento di Matematica,
Via della Ricerca Scientifica, 00133 Roma, Italy.}
\email{digennar@axp.mat.uniroma2.it}

\author{Davide Franco }
\address{Universit\`a di Napoli
\lq\lq Federico II\rq\rq, Dipartimento di Matematica e
Applicazioni \lq\lq R. Caccioppoli\rq\rq, P.le Tecchio 80, 80125
Napoli, Italy.} \email{davide.franco@unina.it}

\abstract Let $C\subset \mathbb P^r$ be an integral projective
curve. We define the speciality index $e(C)$ of $C$ as the
maximal integer $t$ such that $h^0(C,\omega_C(-t))>0$, where
$\omega_C$ denotes the dualizing sheaf of $C$. In the present
paper we consider  $C\subset \mathbb P^5$ an integral degree $d$
curve and we denote by $s$ the minimal degree for which there
exists a hypersurface of degree $s$ containing $C$. We assume that
$C$ is contained in two smooth hypersurfaces $F$ and $G$, with
$deg(F)=n>k=deg (G)$. We assume additionally that $F$ is
Noether-Lefschetz general, i.e. that the $2$-th N\'eron-Severi
group of $F$ is generated by the linear section class. Our main
result is that in this case the speciality index is bounded as
$e(C)\leq {\frac{d}{snk}}+s+n+k-6.$ Moreover equality holds if and
only if $C$ is a complete intersection of   $T:=F\cap G$ with
hypersurfaces of degrees $s$ and ${\frac{d}{snk}}$.

\bigskip\noindent {\it{Keywords}}: Complex projective
curve; speciality index; arithmetic genus;  linkage; Castelnuovo -
Halphen Theory.

\medskip\noindent {\it{MSC2010}}\,: Primary 14N15, 14H99,
14M10; Secondary 14M06, 14N30.

\endabstract
\maketitle

\bigskip
\section{Introduction}
Let $C\subset \mathbb P^r$ be an integral projective curve. We
define the \textit{speciality index} $e(C)$ of $C$ as the maximal
integer $t$ such that $h^0(C,\omega_C(-t))>0$, where $\omega_C$
denotes the dualizing sheaf of $C$. The speciality index of a
space curve is a fundamental invariant which turned out to be
crucial in many issues of projective geometry. For instance, in
the papers \cite{JAG}, \cite{EFGI} and \cite{EFGII}, such an
invariant has been proved to be very useful in the study of
projective manifolds of codimension 2.

In \cite{GP} Gruson and Peskine prove the
following theorem concerning the speciality index of an integral
space curve (see also \cite{GPI}):
\begin{theorem}[Speciality Theorem]
Let $C\subset \mathbb P^3$ be an integral degree $d$ curve not
contained in any surface of degree $< s$. Then  we have:
$$
e(C)\leq {\frac{d}{s}}+s-4.
$$
Moreover equality holds if and only if $C$ is a complete
intersection of surfaces of degrees $s$ and ${\frac{d}{s}}$.
\end{theorem}
In our previous work \cite{SpecialityI}, we prove an extension
of this theorem to curves in $\Pcq$:
\begin{theorem}\cite[Theorem B]{SpecialityI}
Let $C\subset \mathbb P^5$ be an integral degree $d$ curve not
contained in any surface of degree $< s$, in any threefold of
degree $<t$, and in any fourfold of degree $<u$. Assume $d\gg s
\gg t \gg u\geq 1$. Then we have:
$$
e(C)\leq {\frac{d}{s}}+{\frac{s}{t}}+{\frac{t}{u}}+u-6.
$$
Moreover equality holds if and only if $C$ is a complete
intersection of hypersurfaces of degrees $u$, ${\frac{t}{u}}$,
${\frac{s}{t}}$ and ${\frac{d}{s}}$.
\end{theorem}
\par\noindent
Unfortunately, it seems hard to find a generalization of
Gruson-Peskine Speciality Theorem without the assumptions $d\gg
s\gg t\gg u\geq 1$ and to prove a sharp version of the Speciality
Theorem for curves in $\Pcq$.

In this paper we adopt a somewhat different strategy and  prove
a sharp version of the Speciality Theorem for curves in $\Pcq$ under
the assumption that the curve is contained in a smooth
hypersurface with a nice behaviour from the point of view of
Noether-Lefschetz theory (compare with Definition
\ref{NLGeneral}). More precisely, what we are going to do is to
assume that $C$ is contained in a smooth hypersurface having the
$2$-th N\'eron-Severi group  generated by the linear section
class. The main results of this paper are collected in the
following Theorem.

\begin{theorem}
\label{Main} Let $C\subset \mathbb P^5$ be an integral degree $d$
curve. Assume that $C$ is contained in two smooth hypersurfaces
$F$ and $G$, with $deg(F)=n>k=deg (G)$. Assume additionally that
$F$ is Noether-Lefschetz general, i.e. that the $2$-th
N\'eron-Severi group of $F$ is generated by the linear section
class.
\begin{enumerate}
\item
If $C$ is not contained in any hypersurface of degree $<s$,  then
we have:
$$
e(C)\leq {\frac{d}{snk}}+s+n+k-6.
$$
\item If $C$ is  contained in a hypersurface of degree $s<k$,  then the inequality above  still holds true.
Moreover, equality holds if and only if $C$ is a complete
intersection of   $T:=F\cap G$ with hypersurfaces of degrees $s$ and ${\frac{d}{snk}}$.
\end{enumerate}
\end{theorem}

Theorem \ref{Main} turned out to be  a consequence of  much more
general results stated in Theorem \ref{main} and Theorem
\ref{main1}. They show that a sort of Speciality Theorem holds
true for Cohen-Macaulay subschemes $X\subset T$ of codimension 2
in any arithmetically Cohen-Macaulay and factorial variety $T$.
\bigskip
\section{Notations and preliminary results}
\bigskip
In order to prove Theorem \ref{Main}, in this section we gather
some known properties and results, mainly borrowed from our
previous works \cite{CCM}, \cite{Monodromy} and \cite{RCMP}.
\begin{notations}
\label{notation1}
 Let $X\subset \Pn$ be a smooth complete intersection of dimension $2i\geq2$.
Denote by $NS_{i}(X;\mathbb Z)$ be the  $i$-th N\'eron-Severi
group of $X$, i.e. the image of the cycle map:
$$NS_{i}(X;\mathbb Z):= Im(A_{i}(X)\to
H_{2i}(X;\mathbb{Z})\cong H^{2i}(X;\mathbb{Z})).$$
\end{notations}
\vskip3mm

\begin{definition}\label{NLGeneral}
 Let $X\subset \Pn$ be a smooth complete intersection of dimension $2i\geq2$.
 By the Lefschetz hyperplane section theorem we know that the homology group
$ H_{2i}(X;\mathbb{Z})\cong H^{2i}(X;\mathbb{Z}))$ is free. We
will say that $X$ is \textit{Noether-Lefschetz general} if the
rank of $NS_{i}(X;\mathbb Z)$ is one. In this case, the Lefschetz
hyperplane section theorem also implies  that $NS_{i}(X;\mathbb Z)$
is generated by the linear section class $H^i$.
\end{definition}
\par\noindent
In  \cite{CCM}, it can be found a proof for the following result:
\begin{theorem}\cite[Theorem 1.1]{CCM}
\label{CCM} Let $F$ and $G$ be smooth hypersurfaces in
$\Ps^{2m+1}$, with $deg(F)=n>k=deg (G)$, and set $X=F\cap G$. If
$F$ is Noether-Lefschetz general  then $rkNS_m(X)=1$, and $NS_m(X)$ is generated by
the linear section class.
\end{theorem}

\vskip4mm

\begin{notations}
\label{notation2}\begin{enumerate}
\item Let  $Q\subseteq\Pn$ be an irreducible,
reduced, non-degenerate projective variety of dimension $m+1$, with isolated singularities. Let $Q_t$ be a general hyperplane
section of $Q$. Let $U\subset \Ph ^n$ be the  open subset
parametrizing smooth hyperplane sections of $Q$. The fundamental
group $\pi_1(U)$ acts via monodromy on both
$H^{m}(Q_t;\mathbb{Z})$ and $H^{m}(Q_t;\mathbb{Q})$. We  denote by
$$H^{m}(Q_t; \mathbb{Q})=\bI \perp \bV$$  the orthogonal
decomposition given by the monodromy action on the cohomology of
$Q_t$, where $\bI$ denotes the invariant subspace.
\item
Denote by $$i^{\star}_{k}:H_{k+2}(Q;\mathbb Z)\to
H^{2n-k}(Q_t;\mathbb Z)$$ the map obtained composing the Gysin map
$H_{k+2}(Q;\mathbb Z)\to H_{k}(Q_t;\mathbb Z)$ with Poincar\'e
duality $H_{k}(Q_t;\mathbb Z)\cong H^{2n-k}(Q_t;\mathbb Z)$.
\end{enumerate}
\end{notations}
\vskip3mm
\par\noindent
In \cite{Monodromy},  it can be found a proof for the following results:

\begin{theorem}\cite[Theorem 3.1]{Monodromy}
\label{Monodromy} With notations as in \ref{notation2},
 the vector subspace $\bV    \subset H^{m}(Q_t; \mathbb{Q})$ is generated, via monodromy, by standard vanishing
cycles.
\end{theorem}
\vskip3mm

\begin{corollary}\cite[Corollary 3.7]{Monodromy}
\label{corollario}  The vector subspace $\bV\subset H^{m}(Q_t; \mathbb{Q})$ is irreducible via monodromy action.
\end{corollary}
\vskip3mm

The results  \ref{Monodromy} and \ref{corollario} concern rational cohomology. In the paper
\cite{RCMP} they are used to prove similar results concerning integral cohomology:
\begin{theorem}\cite[Theorem 2.1]{RCMP}
\label{RCMP} With notations as in \ref{notation2},
 the following properties hold true.
\begin{enumerate}
\item For any  integer $m< k\leq 2m$ the map $i^{\star}_{k}$ is an
isomorphism, the map
 $i^{\star}_{m}$ is injective  with torsion-free cokernel, and $H_{m+2}(Q;\mathbb
Z)\cong \bI$ via  $i^{\star}_{m}$.
\item For any  even integer $m< k=2i\leq 2m$ the map
$i^{\star}_{k}\otimes_{ \mathbb{Z}} \mathbb{Q}$ induces an
isomorphism $$NS_{i+1}(Q;\mathbb Q)\cong NS_{i}(Q_t;\mathbb Q).$$
\item If $k=2i=m$ and the orthogonal complement $\bV$ of
$\bI\otimes_{\mathbb Z}\mathbb Q$ in $H^{n}(Q_t;\mathbb Q)$ is not
of pure Hodge type $(m/2,m/2)$, then $NS_{i}(Q_t;\mathbb
Z)\subseteq \bI$, and the map $i^{\star}_{n}\otimes_{ \mathbb{Z}}
\mathbb{Q}$ induces an isomorphism $NS_{i+1}(Q;\mathbb Q)\cong
NS_{i}(Q_t;\mathbb Q)$.
\end{enumerate}
\end{theorem}
\vskip3mm
\par\noindent
One of the main ingredients of the proof of Theorem \ref{Main} is the following Proposition.

\begin{proposition}\label{prop} Let $F\subset \Pcq$ be a Noether-Lefschetz general hypersurface and let $G\subset\Pcq $ be
a smooth hypersrface with $k:=degG<d:=degF$.
Define $T:=F\cap G$. Then we have:
\begin{enumerate}
\item the threefold $T$ is factorial with isolated singularities;
\item if $deg\,T\geq 4$ then the general hyperplane section $S:=H\cap T$ is a Noether-Lefschetz general surface.
\end{enumerate}
\end{proposition}
\begin{proof}
(1) The threefold $T$ has at worst finitely many singularities by \cite[Proposition 4.2.6]{Vogel}. Furthermore, $T$ is
 factorial by
Theorem \ref{CCM}.
\par\noindent(2) Combining Theorem \ref{Monodromy}, Corollary \ref{corollario} and Theorem \ref{RCMP}, the proof runs similarly as the  classical one
(compare with the proof of \cite[Theorem 3.2]{IJM}). Indeed,
denote by $U\subset \Ph ^5$ the affine open subset parametrizing
smooth hyperplane sections of $T$. The fundamental group
$\pi_1(U)$ acts via monodromy on both $H^{2}(S;\mathbb{Z})$ and
$H^{2}(S;\mathbb{Q})$.  As in \ref{notation2}, consider the
orthogonal decomposition $H^{2}(S;\mathbb{Q})=\bI \perp \bV$,
where $\bI$ is the $\pi_1(U)$-invariant cohomology (compare also
with  \cite[Notations 3.1 (ii)]{IJM}). By Theorem \ref{Monodromy}
and Corollary \ref{corollario} we know that  the vanishing
cohomology $\bV$ is a $\pi_1(U)$-irreducible module generated by
the standard vanishing cycles. On the other hand,  Theorem
\ref{RCMP} implies that the $\pi_1(U)$-invariant part of
$H^{2}(S;\mathbb{Z})\simeq H_{2}(S;\mathbb{Z})$ is the image of
the  Gysin map:
\begin{equation}
\label{Gysin}
\bI \cap H_{2}(S;\mathbb{Z})= Im( H_{4}(T;\mathbb{Z})\stackrel{\cap \,u}\longrightarrow H_{2}(S;\mathbb{Z}))
\end{equation}
(here $u\in H^2(T, T-S;\mathbb{Z})$ denotes the orientation class
\cite[\S 19.2]{FultonIT}). By point (1) $T$ is factorial, hence
the subspace $\bI$ is generated by the hyperplane class. But then
$\bV$ is not of pure Hodge type because $deg\,T\geq 4$. By
irreducibility, the image of $NS_{1}(S;\mathbb Z)$ in $\bV$
vanishes. This implies that the N\'eron-Severi group
$NS_{1}(S;\mathbb Z)$ is $\pi_1(U)$-invariant and  (\ref{Gysin})
says that $S$ is Noether-Lefschetz general.
\end{proof}
\bigskip
\section{Proof of Theorem \ref{Main} (2)}
\bigskip
\begin{definition}
Let $X$ be a Cohen-Macaulay projective scheme. We
define the \textit{speciality index} $e_X$ of $X$
as the maximal integer
$t$ such that $h^0(X,\omega_X(-t))>0$, where $\omega_X$ denotes
the dualizing sheaf of $X$.
\end{definition}

The proof of Theorem \ref{Main} (2) rests on the following much more general result:

\begin{theorem}[Speciality theorem for aCM varieties]\label{main}
Let $T\subset \Pn$ be an arithmetically Cohen-Macaulay (aCM for
short), factorial and subcanonical variety with $dimT=m\geq 3$ and
$\omega_T\simeq \oc_T(t)$.\par\noindent Let $G\subset T$ be an
integral  divisor. Since $T$ is factorial and aCM, we have $G= \tg
\cap T$ with $\tg\subset \Pn$ a projective hypersurface of some
degree $g$. Let $X\subset G$ be a  Cohen-Macaulay scheme of
codimension two in $T$.\par\noindent Then $$e_X\leq
\frac{deg(X)}{deg(T)g} +g+t$$ and the equality holds iff $X$ is a
complete intersection $X=T\cap \tilde{G}\cap H$, with
$deg(H)=\frac{deg(X)}{deg(T)g}$.
\end{theorem}

\begin{proof}
Consider a general hypersurface $P$ of degree $p\gg 0$ containing
$X$. Denote by $Y$ the scheme $T\cap \tg \cap P$ which we are
going to consider as a complete intersection in $T$. Following
Peskine-Szpiro \cite{PS}, we consider the scheme $R$ residual  of
$X$ with respect to $Y$ (compare also with \cite[\S 2]{FKL}).
\par The   \textit{Noether Linkage Sequence} \cite[Proposition 2.3]{FKL} inside $T$ looks like
\begin{equation}\label{NLS'}
0\to \ic_Y \to \ic _R \to \omega_X \otimes \omega_Y^{-1}\to 0,
\end{equation}
and can be written as
\begin{equation}\label{NLS}
0\to \ic_Y \to \ic _R \to \omega_X (-t-g-p)\to 0
\end{equation}
(all the ideal sheaves are meant to be defined in $T$).
Recall that
\begin{equation}
\label{not0}
h^0(\omega_X(-e))\not =0
\end{equation}
($e:=e_X$). Since $T$ is aCM and $Y$ is a complete intersection in
$T$ of type $(g,p)$, the short exact sequence
$$0\to  \oc _T(-g-p) \to \oc _T(-g) \oplus  \oc _T(-p)\to \ic_Y\to 0
$$
implies
$$\dots\to  H^1( \oc _T(l-g) \oplus  \oc _T(l-p))\to H^1(\ic_Y(l))\to H^2(\oc _T(l-g-p)) \to \dots \hskip2mm \forall l  $$
hence
\begin{equation}
\label{=0}
h^1(\ic_Y(t+g+p-e))=0.
\end{equation}
Combining (\ref{NLS}), (\ref{not0}) and (\ref{=0}), we see that  there exists a hypersurface $S$ of degree
$s=t+g+p-e$ containing $R$ and not containing $Y$. But $G$ is integral and
$Y'=G\cap S$ is a complete intersection, in $T$, containing $R$. Set $Y'=R\cup R'$ the corresponding,
possibly algebraic, linkage.
But then
$$deg(R')+deg(R)=deg(T)gs, \hskip2mm deg(X)+deg(R)=deg(T)gp$$ and by a simple computation, we find
$$deg(R')=deg(X)-deg(T)g(e-t-g)\geq 0$$ and the first statement follows.
\par
Suppose now the equality holds. Then we have  $$deg(X)=deg(T)g(e-t-g)=deg(T)g(p-s).$$
and the scheme $R'$ is empty.
Furthermore, we have that $R=Y'=G\cap S$ is a complete intersection with
$\omega _R\simeq \oc_R(t+g+s)$.
\par\noindent
Coming back to the Noether Linkage Sequence (\ref{NLS'})
$$
0\to \ic_Y \to \ic _X \to \omega_R \otimes \omega_Y^{-1}\to 0
$$
we find
\begin{equation}\label{NLSlast}
0\to \ic_Y \to \ic _X \to \oc _R (s-p)\to 0.
\end{equation}
Similarly as above, the short exact sequence
$$0\to  \oc _T(-g-s) \to \oc _T(-g) \oplus  \oc _T(-s)\to \ic_R\to 0
$$
implies
$$\dots\to  H^1( \oc _T(l-g) \oplus  \oc _T(l-s))\to H^1(\ic_Y(l))\to H^2(\oc _T(l-g-s)) \to \dots \hskip2mm \forall l  $$
and
$$h^1(\ic_R(p-s))=0.$$
Hence there is a hypersurface $H$ of degree $h=p-s$ containing $X$ and not containing $Y$.
Finally, since $G$ is integral and $deg(X)=deg(T)g(p-s)$ we conclude that $X=G\cap H$.
\end{proof}

\vskip3mm
\begin{proof}[Proof of Theorem \ref{Main} (2)]
It suffices to apply Theorem \ref{main} to the complete
intersection $T:=F\cap G$, which is aCM with $dimT= 3$ and
$\omega_T\simeq \oc_T(n+k-6)$, and  factorial in view of
Proposition \ref{prop}.
\end{proof}
\vskip3mm
\section{Proof of Theorem \ref{Main} (1)}
\bigskip
The proof of Theorem \ref{Main} (1) rests on the following much more general result:

\begin{theorem}
\label{main1} Let $T\subset \Pn$ be an aCM, factorial variety with
$dimT=m\geq 3$ and $\omega_T\simeq \oc_T(t)$. Assume moreover that
$T$ is smooth in codimension $2$ and that the very general surface
section of $T$ is factorial.  Let $X\subset T$ be a C.M. subscheme
of codimension $2$ which is generically complete intersection.  If
$h^0(\ic_{X,T}(h-1))=0$ and $h>0$ then
$$
e_X\leq \frac{deg(X)}{deg(T)h} +h +t.
$$
\end{theorem}
\vskip2mm
\par\noindent
The main idea in the proof of \ref{main1}, which goes back to the work of Hartshorne, is to construct a rank two relexive sheaf on $T$ having a section vanishing in $X$ (see e.g. \cite{Hrs} and \cite{BF}).
\par\noindent
In order to prove Theorem \ref{main1}, we need some preliminary results.
We recall the following result of R. Hartshorne:
\begin{lemma}\cite[Proposition 1.3]{Hrs}
\label{reflexive}
Let $T$ be a normal scheme and let $\fc $ be a coherent sheaf defined on $T$. Then $\fc $ is reflexive iff
\begin{enumerate}
\item $\fc $ is torsion-free;
\item $\forall x\in T$, $dim \oc_x\geq 2 \Longrightarrow depth\fc _x \geq2$.
\end{enumerate}
\end{lemma}

For the sake of completeness, we give a short proof of the following (maybe well known) result.
\begin{lemma}
\label{ext} Let $T\subset \Pn$ be an aCM  scheme such that
$m:=dimT\geq 3$ and $\omega_T\simeq \oc_T(t)$. Let $X\subset T$ be
a Cohen-Macaulay     subscheme of codimension $2$.
\par\noindent Then we have:
$$Ext^1_T(\ic_{X,T}(c), \oc_T)\simeq H^0(X,\omega_X(-c-t)), \hskip2mm \forall c\in \mathbb{Z}.$$
\end{lemma}
\begin{proof}
By applying the functor $Hom_T(\cdot, \oc_T)$ to the short exact sequence
$$
0 \to \ic_{X,T}(c) \to \oc_T(c) \to \oc_X(c) \to 0
$$
we find
\begin{multline*}
Ext^1_T(\oc_{T}(c), \oc_T)\to  Ext^1_T(\ic_{X,T}(c), \oc_T) \to \\
\to Ext^2_T(\oc_{X}(c), \oc_T) \to Ext^2_T(\oc_{T}(c), \oc_T).
\end{multline*}
By Serre Duality,
 $\omega_T\simeq \oc_T(t)$ implies
 $$Ext^i_T(\oc_{T}(c), \oc_T)\simeq H^{m-i}(\oc_T(-c-t))=0, \hskip2mm i= 1, 2$$
where the last equality follows from the hypothesis
  that  $T$ is aCM of dimension $\geq 3$. Again by Serre Duality we have:
\begin{multline*}
Ext^1_T(\ic_{X,T}(c), \oc_T) \simeq Ext^2_T(\oc_{X}(c), \omega_T(-t))\simeq \\
 H^{m-2}(T,\oc _X(c+t))\simeq H^{m-2}(X,\oc _X(c+t))\simeq H^{0}(X,\omega
 _X(-c-t)).
\end{multline*}
\end{proof}

\begin{proposition}
\label{sc} Let $T\subset \Pn$ be an aCM  variety such that
$m:=dimT\geq 3$ and $\omega_T\simeq \oc_T(t)$. We assume
additionally  that $T$ is smooth in codimension $2$. For any pair
$(X, \xi)$ with:
\begin{itemize}
\item $X\subset T$ a Cohen-Macaulay, generically complete intersection subscheme of codimension two in $T$,
\item $\xi \in H^0(\omega_X(-t-c))$ generating almost everywhere,
\end{itemize}
there exists  a rank two reflexive sheaf $\fc $  on $T$, with $c_1(\fc )=cH$, $c_2(\fc )=[X]$
(the fundamental cycle of $X$) and such that
\begin{equation}
\label{extension} 0 \rightarrow \oc_T \rightarrow \fc  \rightarrow
\ic _{X\mid T}(c) \rightarrow 0.
\end{equation}
\end{proposition}

\begin{proof} The assertion concerning the Chern classes follows trivially from the rest of the statement
so it suffices to prove the existence  of a sequence like (\ref{extension}), with $\fc$ reflexive.
\par
The existence of a sequence like (\ref{extension})   follows
directly from Lemma \ref{ext}. Since $T$ is Cohen-Macaulay and
smooth in codimension $2$, it is also normal by Serre's criterion.
Then we may apply Lemma \ref{reflexive} in order to prove the
reflexivity of $\fc$. Further, since $T$ is Cohen-Macaulay, both
$\oc_T$ and $\ic _{X,T}(c)$ are torsion-free hence we  only need
to prove the second condition of Lemma \ref{reflexive}. Fix a
point $x$ of codimension $\geq 3$ and denote by $\bK$ the residue
field at $x$. Applying the functor $Hom_{\oc_x}(\bK , \cdot)$ to
the  sequence
$$0 \rightarrow \ic _{x,X\mid T} \rightarrow \oc_{x,T}  \rightarrow \oc_{x,X} \rightarrow 0
$$
and recalling that both $T$ and $X$ are Cohen-Macaulay we have:
\begin{equation}
\label{ext0}
Ext^i_{\oc_x}(\bK , \ic _{x,X\mid T})=0 , \hskip2mm i\leq 2.
\end{equation}
Applying the functor $Hom_{\oc_x}(\bK , \cdot)$ to the  sequence
$$0 \rightarrow \oc_{x,T} \rightarrow \fc_x  \rightarrow \ic _{x,X\mid T}(c) \rightarrow 0  $$
we see that the vanishing ($\ref{ext0}$)  implies:
 $$\forall x\in X, \hskip2mm dim \oc_x\geq 3 \Longrightarrow depth\fc _x \geq2 .$$
In order to conclude we  need to prove:
$$\forall x\in X, \hskip3mm dim \oc_x= 2 \Longrightarrow depth\fc _x \geq2.
$$
 What we are going to do is to prove that
$\fc _x$ is a free  module of rank two over $\oc _x$, for any
$x\in X$ such that $dim \oc_x= 2$. In order to do this, we prove
that $\fc _x$ has homological dimension $0$ (\cite[IV]{Serre}).
Since $T$ is smooth in codimension $2$, $\forall x\in X$ of
codimension $2$ the local ring $\oc _x$ is regular of dimension
$2$. So it suffices to prove that
\begin{equation}
\label{suffices} \ec xt^1_{T}(\fc ,\oc_T )_x=\ec xt^2_{T}(\fc
,\oc_T )_x=0.
\end{equation}
From the sequence (\ref{extension}) we see that $dh(\fc _x)\leq
dh(\ic _{X\mid T})=1$, the first inequality coming from (\cite[IV
p. 28]{Serre}) and  the last equality coming from the fact that
$\ic _{X\mid T}$ is complete intersection at $x$. So, in order to
prove (\ref{suffices}) we are left to show that $\ec xt^1_{T}(\fc
,\oc_T )_x=0$. Applying $\hc om_T(\cdot , \oc_T(c))$ to the
sequence ($\ref{extension}$) we get:
\begin{equation}
\label{se5}
0 \to \oc_T \to \fc^* (c)\to \oc_T \stackrel{\xi }{\to }\omega_X(-t) \to
\ec xt^1_{T}(\fc ,\oc_T(c) )\to 0
\end{equation}
where we have taken into account the isomorphism $\ec xt^1_{T}(\ic
_{X,T},\oc_T)\simeq \omega_X(-t)$, which again follows from the
fact that both $T$ and $X$ are Cohen-Macaulay and $\omega_T\simeq
\oc_T(t)$. Since $T$ is smooth in codimension $2$, $\forall x\in
X$ of codimension $2$ the local ring $\oc _x$ is regular of
dimension $2$. Furthermore, since $\xi $  generates almost
everywhere and $X$ is generically complete intersection, the
fourth map of the  sequence (\ref{se5}) is an isomorphism at $x$
hence $\ec xt^1_{T}(\fc ,\oc_T(c) )_x\simeq 0$ and $\fc _x$ is a
free module of rank two over $\oc_x$.
\end{proof}

\begin{lemma}
\label{1section} Let $C\subset \Pn $ be a smooth variety and $E$ a
rank two vector bundle on $C$ having a section vanishing in the
right dimension. If $c_1(E)< 0$ then $h^0(E)=1$ and
$h^0(E(-m))=0$, $\forall m>0$.
\end{lemma}

\begin{proposition}
\label{bogomolov} Let $T\subset \Pn$ be an aCM, factorial  variety
such that $m:=dimT\geq 3$ and $\omega_T\simeq \oc_T(t)$. We assume
additionally  that $T$ is smooth in codimension $2$ and  that the
general hyperplane section of $T$ is factorial.  Let $\fc $ be a
normalized (i.e. with $-1\leq c_1(\fc)\leq 0$) reflexive sheaf on
$T$. If $d(c_1(\fc )\cdot c_1(\fc)) > 4d(c_2(\fc ))$ then there
exists $\alpha \leq 0$ such that $h^0(\fc (\alpha))\not =0$.
Furthermore, if $c_1(\fc )=0$ then $\alpha <0$ hence    we have
$c_1(\fc (\alpha))< 0$.
\end{proposition}
\begin{proof} Let us denote by $S$ the general (smooth)
surface section of $T$. Since $\fc\mid_S$ is a normalized rank $2$
vector bundle on $S$, Bogomolov's theorem implies there exists
$\alpha \leq 0$ such that $h^0(S,\fc(\alpha) \mid_S )\not =0$.
Moreover, we can assume $\alpha <0$ as soon as $c_1(\fc\mid
_S)=0$. Bogomolov's theorem implies that a section of $\fc\mid _S
(\alpha)$  can be chosen in such a way that it vanishes in the
right dimension. In any case we have $c_1(\fc\mid _S (\alpha))<0$,
so Lemma \ref{1section} above implies $h^0(S,\fc(\alpha) \mid_S
)=1$.
\par Fix $C\subset S$ a general curve section of $T$. We can
assume that $C$ does not meet the zero locus of the general
section of $\fc\mid _S (\alpha)$ so Lemma \ref{1section} implies:
\begin{equation}
\label{restrictiontoC}
h^0(C,\fc(\alpha) \mid_C )=1 \hskip3mm \textrm{and} \hskip3mm h^0(C,\fc(\beta) \mid_C )=0
\hskip2mm \forall \beta <\alpha.
\end{equation}
Set
$$\pc\simeq \mathbb{P}^{m-2} = \{ \pi \in \mathbb{G} (n-m+2,\Pn ):
\hskip1mm C\subset \pi \} \subset\mathbb{G} (n-m+2,\Pn ),$$
denote by $\tT \subset T\times \pc$ the incidence variety:
$$\tT=\{ (x,\pi) \in T\times \pc: \hskip1mm x\in \pi \cap T \}$$
and by $\phi :\tT \to T$, $\psi :\tT \to \pc$ the natural maps.
\begin{claim}
 $h^0(\psi^{-1}(p), \fc (\alpha)\mid _{\psi^{-1}(p)})=1$, \hskip2mm
$\forall p\in \pc$.
\end{claim}
\par\noindent
As we have just said $h^0(\psi^{-1}(p), \fc (\alpha)\mid
_{\psi^{-1}(p)})=1$ for a very general  $p\in \pc $ so, by
semicontinuity, $h^0(\psi^{-1}(p), \fc (\alpha)\mid
_{\psi^{-1}(p)})\geq 1$,  $\forall p\in \pc$. In order to prove
the Claim it is then sufficient to prove that $h^0(\psi^{-1}(p),
\fc (\alpha)\mid _{\psi^{-1}(p)})< 2$,  $\forall p\in \pc$. Set
$S'=\psi^{-1}(p)$ and assume by contradiction $h^0(S', \fc
(\alpha)\mid _{S'})\geq 2$. From the short exact sequence
$$
0\to \fc \mid_{S'}(\alpha -1) \to \fc \mid_{S'}(\alpha ) \to \fc \mid_{C}(\alpha ) \to 0
$$
and taking into account ($\ref{restrictiontoC}$) we get $h^0(S',
\fc (\alpha-1)\mid _{S'})\not =0.$ Set
$\overline{\alpha}:=min\{\beta \in \mathbb{N}: \hskip1mm h^0(S',
\fc (\beta)\mid _{S'})\not =0 \}\leq \alpha-1 $. From the short
exact sequence
$$
0\to \fc \mid_{S'}(\overline{\alpha} -1) \to \fc \mid_{S'}(\overline{\alpha} ) \to \fc \mid_{C}(\overline{\alpha} ) \to 0
$$
and by the definition of $\overline{\alpha}$ we find $h^0(C,\fc
\mid_{C}(\overline{\alpha} ))\not=0$ which contradicts
($\ref{restrictiontoC}$) since $\overline{\alpha}<\alpha$. The
claim is so proved.
\par
By Grauert's theorem \cite[Corollary 12.9]{Hartshorne},  $\psi_* \phi^* \fc(\alpha)$ is an invertible sheaf on $\pc $.
On the other hand, since $\phi^{-1}C= C\times \pc $, we have
$$
\psi_*(\phi^*(\fc (\alpha ))\mid _{\phi^{-1}C})\simeq H^0(C,\fc (\alpha)\mid _C)\otimes \oc _{\pc}
\simeq \oc _{\pc}.
$$
Finally, the natural restriction
$H^0(\psi ^{-1}(p), \fc (\alpha)\mid _{\psi ^{-1}(p)})\to H^0(C, \fc(\alpha )\mid _C)$ is an isomorphism $\forall p\in \pc$,
so the natural map
$\psi_*\phi^*(\fc (\alpha )) \to
\psi_*(\phi^*(\fc (\alpha ))\mid _{p^{-1}C})\simeq  \oc _{\pc}$
is an isomorphism of invertible sheaves on $\pc $. Then we have
$$H^0(\tT,\phi^*(\fc (\alpha ))=H^0(\pc,\psi_*(\phi^*(\fc (\alpha ))= H^0(\pc,\oc_{\pc})=\mathbb{C}.$$
We conclude  by means of the projection formula, because
$\phi_*\oc _{\tT}\simeq \oc_T$.
\end{proof}

\begin{remark}
\label{propalpha}
\begin{enumerate}
\item By  Lemma \ref{1section}, the coefficient $\alpha$
arising in Proposition \ref{bogomolov} is the least twist of $\fc$ admitting a section.
\item The proof of Proposition \ref{bogomolov} shows that  the zero locus of the section of $\fc(\alpha)$
has the right dimension, because it does not meet the general curve $C$.
\end{enumerate}
\end{remark}

\begin{proof}[Proof of Theorem \ref{main1}.]
In this proof we closely follow \cite{RV}.

By Proposition \ref{sc} there exists a normalized reflexive sheaf $\fc $ (on $T$) such that
$$
0 \to \oc_T \to \fc(k) \to \ic _X(e-t) \to 0
$$
($c_1(\fc )=cH$, $c_2(\fc )=[X]-(ck+k^2)H^2$ and $c+2k=e-t$).
Set
\begin{itemize}
\item let $\alpha $ and $\beta $ be the smallest degrees of two independent generators of $H^0_*\fc$
(compare with \cite[p. 103]{RV}) ,
\item $s=min \{r: \hskip1mm h^0(\ic_{X,T}(r))\not =0 \}$.
\end{itemize}
We distinguish two cases depending on whether  the discriminant of $\fc $ is $\leq 0$ or $>0$.
\par $d(c_1(\fc )\cdot c_1(\fc)) \leq 4d(c_2(\fc ))$. This case is the simplest one because the expression
$d(X)-d(T)h(e'-h-t)$ is the degree of the second Chern class of
$\fc(k-h)$. Since the discriminant is $\leq 0$, the second Chern
class is always positive and we are done. \par $d(c_1(\fc )\cdot
c_1(\fc)) > 4d(c_2(\fc ))$. In this case Proposition
\ref{bogomolov} implies $\alpha\leq 0$ ($<0$ if $c=0$).
Furthermore, Remark \ref{propalpha} (2) says that the
corresponding section vanishes in the right dimension. Then
$d(c_2(\fc (\alpha))=d(c_2(\fc (-\alpha-c))\geq 0$ and the degree
of the  second Chern class is positive for any twist $\leq \alpha
$ or $\geq -\alpha-c$. If $k=\alpha $ then $s=\beta +\alpha +c$
and the expression $d(X)-d(T)h(e'-h-t)$ is the degree of
$c_2(\fc(k-h))=c_2(\fc (h-\alpha-c))$ which is strictly positive
since $h>0$. So the inequality is proved and the equality cannot
be attained. If $k\geq \beta $ then $s= \alpha +k+c$ and the
expression $d(X)-d(T)h(e'-h-t)$ is the degree of
$c_2(\fc(k-h))=c_2(\fc (\alpha-(s-h)))$. So the inequality is
proved and the equality can be attained only if $s=h$ and the
degree of $c_2(\fc (\alpha))$ vanishes. \end{proof}

\vskip3mm
\begin{proof}[Proof of Theorem \ref{Main} (1)]
It suffices to apply Theorem \ref{main1} to the complete
intersection $T:=F\cap G$, which is aCM with $dimT= 3$ and
$\omega_T\simeq \oc_T(n+k-6)$. The hypotheses that $T$ is
factorial and smooth in codimension $2$ and that the very general
surface section of $T$ is factorial follow from Proposition
\ref{prop}.
\end{proof}

\end{document}